\documentclass[11pt,a4paper]{amsart}
\usepackage{graphicx,amssymb}
\input xy
\xyoption{all}
\usepackage[all]{xy}
\usepackage{hyperref}
\usepackage{mathdots}
\usepackage{nomencl}
\makeindex
\makenomenclature

\newtheorem{thm}{Theorem}[section]
\newtheorem{cor}[thm]{Corollary}
\newtheorem{lem}[thm]{Lemma}
\newtheorem{prop}[thm]{Proposition}
\theoremstyle{definition}

\newtheorem{rem}[thm]{Remark}
\newtheorem{exam}[thm]{Example}
\numberwithin{equation}{section}

\newcommand{\RR}{\mathbb R}
\newcommand{\QQ}{\mathbb Q}

\newcommand{\ZZ}{\mathbb Z}
\newcommand{\CC}{\mathbb C}

\newcommand{\hra}{\hookrightarrow}
\newcommand{\ra}{\rightarrow}

\newcommand{\cA}{\mathcal{A}}

\newcommand{\cL}{\mathcal{L}}

\newcommand{\cH}{\mathcal{H}}

\newcommand{\cE}{\mathcal{E}}
\newcommand{\cG}{\mathcal{G}}

\newcommand{\cO}{\mathcal{O}}

\newcommand{\cF}{\mathcal{F}}

\DeclareMathOperator{\Hom}{{Hom}}

\DeclareMathOperator{\Pic}{Pic}

\DeclareMathOperator{\diag}{{diag}}
\DeclareMathOperator{\End}{{End}}
\DeclareMathOperator{\NS}{{NS}}

\begin{document}

\title{A criterion for an abelian variety to be non-simple}

\author{R. Auffarth, H. Lange, A. M. Rojas}

\address{R. Auffarth \\Departamento de Matem\'aticas, Facultad de
Ciencias, Universidad de Chile, Santiago\\Chile}
\email{rfauffar@uchile.cl}

\address{H. Lange\\Mathematisches Institut,
              Universit\"at Erlangen-N\"urnberg\\Germany}
\email{lange@mi.uni-erlangen.de}

\address{A. M. Rojas \\Departamento de Matem\'aticas, Facultad de
Ciencias, Universidad de Chile, Santiago\\Chile}
\email{anirojas@u.uchile.cl }


\thanks{Partially supported by Fondecyt Grants 3150171 and 1140507 and CONICYT PIA ACT1415.}%
\subjclass[2010]{14K02; 14K12; 32G20}%
\keywords{Abelian variety, numerical class, N\'eron-Severi group.}%

\maketitle

\begin{abstract}
We give a criterion in terms of period matrices for an arbitrary polarized abelian variety to be non-simple. Several examples are worked out.
\end{abstract}
\maketitle

\section{Introduction} 

Let $(A,\cL)$ be a complex abelian variety of dimension $g$ with polarization of type 
$D = \diag(d_1,\dots,d_g)$ defined by an ample line bundle $\cL$. So $A = V/\Lambda$   
where $V$ is a complex vector space of dimension $g$ and $\Lambda$ is a lattice of maximal 
rank in $\CC^g$   such that with respect to a basis of $V$ and a 
symplectic basis of $\Lambda$, $A$ is given by a period matrix $(D\; Z)$ with $Z$ in the Siegel 
upper half space of rank $g$. The aim of this paper to give a set of equations in the entries of the 
matrix $Z$ which characterize the fact that $(A,\cL)$ is non-simple. These equations are easy to 
work out for $g=2$ and can be given explicitly with the help of a computer program 
for $g = 3$.

To be more precise, the polarization induces a bijection 
$$
\varphi: \NS_{\QQ}(A) \ra 
\End_\QQ^s(A)
$$ 
of the rational N\'eron-Severi group 
$\NS_{\QQ}(A) := \NS(A) \otimes \QQ
= \left(\Pic(A)/\Pic^0(A)\right) \otimes \QQ$
with the $\QQ$-vector space $\End_{\QQ}^s(A) := \End^s(A) \otimes \QQ$ generated by the endomorphisms of 
$A$ which are symmetric with respect to the Rosati involution of $(A,\cL)$. Now an abelian 
subvariety $X$ of $A$ corresponds to a symmetric idempotent $\varepsilon_X \in \End^s_{\QQ}(A)$. So $\varphi^{-1}(\varepsilon_X)$ is an element of 
$\NS_{\QQ}(A)$. On the other hand, $\NS_{\QQ}(A)$ admits an intersection product 
which associates to $g$ elements $\alpha_1, \dots, \alpha_g \in \NS_{\QQ}(A)$ a rational number $(\alpha_1 \cdots \alpha_g)$. Theorem \ref{thm3.3} is a criterion for an element 
$\alpha \in \NS_{\QQ}(A)$ to be equal to $\varphi^{-1} \varepsilon_X$ for an abelian subvariety $X$ of $A$
in terms of the intersection numbers $(\alpha^r \cdot \cL^{g-r})$.

Introducing coordinates of $\Lambda$ as above and using the fact that 
$$
\NS_{\QQ}(A) = H^{1,1}(A)\cap H^2(A,\QQ)
$$ 
we translate the criterion into terms of
differential forms which finally gives the above mentioned equations in Theorem \ref{thm3.1} for the 
matrix $Z$. These have been outlined in \cite{robert} in the case of a principally polarized abelian 
variety. For our applications we need however the generalization  to an arbitrary polarized abelian 
variety as we will explain now.

Let $G$ be a finite group acting faithfully on the abelian variety $A$. Following \cite[Section 13.6]{bl}, this action induces a morphism $\rho$ from
the group algebra $\mathbb{Q}[G]$ to the rational endomorphism algebra $\text{End}(A)\otimes_{\mathbb{Z}}\mathbb{Q}$. Since $\mathbb{Q}[G]$ is
a semisimple algebra, it decomposes as a product of simple algebras $Q_1\times \dots \times Q_r$. Each $Q_i$ is generated by a central idempotent $e_i$, and
these are in correspondence with the rational irreducible representations of $G$. By defining $A_i=\rho(me_i),$ where $m$ is an integer such that $\rho(me_i)\in \text{End}(A)$, the so called isotypical decomposition of $A$ is obtained. It is an isogeny
$ A_1 \times \cdots \times A_r \ra A$ where the $A_i$ are abelian subvarieties of $A$  uniquely 
determined by the simple factors $Q_i$ of the rational group algebra $\QQ[G]$. 

In an analogous way the factors $A_i$ are decomposed further up to isogeny as $A_i \sim B_i^{n_i}$. This last decomposition for each isotypical factor comes from the decomposition of $Q_i$ as a product of minimal ideals. Therefore here $B_i$ is an abelian subvariety 
of $A_i$, not uniquely determined,  and $n_i = \frac{\deg \chi_i}{m_i}$, where $\chi_i$ is a
complex irreducible representation associated to the simple factor $Q_i$ and $m_i$ its Schur index
(see \cite[Section 13.6]{bl}). The decomposition
\begin{equation}
A \sim B_1^{n_1} \times \cdots \times B_r^{n_r}
\end{equation}
is called the {\it group algebra decomposition} of the $G$-abelian variety $A$. 
Our starting point was the question whether the abelian varieties $B_i$ are simple. 
Even if $A$ is principally polarized, the induced polarization on $B_i$ is in general not principal. So 
in order to discuss the simplicity of $B_i$ we need Theorem \ref{thm3.1} also in the  non-principally 
polarized case. We will outline several examples for this.

In Section 2 we recall and outline some more details about the relation between abelian subvarieties 
and symmetric idempotents of a polarized abelian variety. Section 3 contains the above criterion in 
terms of the N\'eron-Severi group and Section 4 its translation in terms of period matrices. Finally
Section 5 contains the examples.\\

\noindent\textit{Acknowledgements:} We would like to thank Pawe\l \hspace{0.04cm} Bor\'owka for pointing out a mistake in a previous version of Example \ref{Borowka}. 

\renewcommand{\nomname}{List of Symbols}

\nomenclature[01]{$A$}{complex abelian variety}
\nomenclature[02]{$\cL$}{ample line bundle on $A$}
\nomenclature[03]{$(d_1,\ldots,d_g)$}{type of the line bundle $\cL$}
\nomenclature[04]{$X$}{abelian subvariety of $A$}
\nomenclature[05]{$e_X$}{exponent of $X$}
\nomenclature[06]{$N_X$}{norm endomorphism of $X$}
\nomenclature[07]{$'$}{Rosati involution}
\nomenclature[08]{$\mbox{End}^s_\QQ(A)$}{$\QQ$-endomorphisms of $A$ fixed by $'$}
\nomenclature[09]{$\mbox{NS}_\QQ(A)$}{N\'eron-Severi group of $A$ tensored with $\QQ$}
\nomenclature[10]{$\delta_X$}{numerical class associated to $X$ in $\mbox{NS}_\QQ(A)$}
\nomenclature[11]{$\alpha_X$}{$\frac{1}{e_X}\delta_X$}
\nomenclature[12]{$\varepsilon_X$}{symmetric idempotent associated to $X$}
\nomenclature[13]{$(\alpha_1\cdots\alpha_g)$}{intersection number of numerical classes}
\nomenclature[14]{$V$}{tangent space of $A$ at 0}
\nomenclature[15]{$\Lambda$}{lattice of $A$ in $V$}
\nomenclature[16]{$Z$}{Riemann matrix of $A$}
\nomenclature[17]{$\sim$}{isogeny relation}

\printnomenclature[6cm]

\section{Abelian subvarieties and symmetric idempotents}

According to \cite[Section 5.3]{bl} there is a 1-1 correspondence between the set of abelian 
subvarieties of a polarized abelian variety and the set of symmetric idempotents of its endomorphism 
algebra. In the sequel we need however some more details of this relation which will be worked out in this section.

Let $(A,\mathcal{L})$ be a polarized abelian variety of dimension $g$ of type $(d_1,\ldots,d_g)$.
Here the $d_i$ are positive integers with $d_i|d_{i+1}$ for for all $i$. 
Moreover, if $\widehat{A}:=\mbox{Pic}^0(A)$ denotes the dual abelian variety and
$\phi_{\mathcal{L}}:A\to \widehat{A}$ is the morphism 
$a\mapsto t_a^*\mathcal{L}\otimes\mathcal{L}^{-1}$, then 
$$
\ker\phi_\mathcal{L}\simeq(\ZZ/d_1\ZZ\oplus\cdots\oplus\ZZ/d_g\ZZ)^2.
$$
Now let $X \subset A$ be an abelian subvariety of dimension $n$.  The restriction $\cL|_X$ 
defines a polarization on $X$ of some type $(e_1,\ldots,e_n)$. We will write $e_X$  instead of $e_n$ 
and call it the \emph{exponent} of $X$ (in $(A,\cL)$). We have the commutative diagram
\begin{equation}  \label{diag2.1}
\xymatrix
{
A \ar[r]^{\phi_{\cL}} & \widehat A \ar[d]^{\widehat \iota_X}\\
X\ar[r]^{\phi_{\cL|_X}} \ar[u]^{\iota_X}&\widehat{X}
}
\end{equation}
where $\iota_X$ is the natural inclusion. Since $\phi_{\mathcal{L}|_X}$ is an isogeny, also 
$$
\psi_X:= e_X \phi_{\cL|_X}^{-1}: \widehat{X}\ra X
$$ 
is an isogeny. The \emph{norm endomorphism} of $X$ (with respect to the polarization $\cL$) is 
defined by the composition
$$
N_X:=\iota_X\psi_X \widehat\iota_X\phi_\mathcal{L}\in\mbox{End}(A).
$$
Clearly the image of $N_X$ is $X$ itself.
Let 
$$
' : \End_{\QQ}(A) \to \End_{\QQ}(A), \qquad \alpha \mapsto \alpha' := \phi_{\cL}^{-1} \widehat \alpha 
\phi_{\cL} 
$$
 be the Rosati involution with respect to the polarization $\cL$ and let $\End_{\QQ}^s(A)$ be the subspace of elements of $\End_{\QQ}(A)$ fixed under it. Clearly $N_X \in \End_{\QQ}^s(A)$ and we have according to
 \cite[Prop. 5.2.1]{bl} that the map 
\begin{equation}  \label{eq2.1}
\varphi:\NS_{\QQ}(A)\to \End_{\QQ}^s(A),  \qquad  \alpha\mapsto\phi_\mathcal{L}^{-1}\phi_\alpha
\end{equation}
is an isomorphism of $\QQ$-vector spaces.
We define the {\it numerical class of $X$ in} $\NS_{\QQ}(A)$ by
$$
\delta_X:=\varphi^{-1}(N_X)\in \NS_{\QQ}(A) .
$$
Apart from $X$ it depends only on the polarization $\cL$ of $A$. Recall that  the symmetric idempotent associated to $X$ is defined as
$$
\varepsilon_X:=\frac{1}{e_X}N_X\in\mbox{End}_\QQ^s(A).
$$
Let $\alpha_X$ denote the class of $\NS_{\QQ}(A)$ corresponding to $\varepsilon_X$ via the 
isomorphism \eqref{eq2.1}, i.e.
$$
\alpha_X = \varphi^{-1}(\varepsilon_X).
$$

\begin{lem}   \label{l2.1}
Let $[N_X^*\mathcal{L}]$ denote the class of the line bundle $N_X^*{\cL}$ in $\NS(A)$. Then we have
$$
\delta_X=\frac{1}{e_X}[N_X^*\mathcal{L}] \qquad \mbox{and}  \qquad 
\alpha_X= \frac{1}{e_X^2} [N_X^*\cL] \in \NS_{\QQ}(A).
$$
\end{lem}
\begin{proof}
Using the symmetry of $N_X$ with respect to Rosati and the fact that $N_X^2 = e_X N_X$ we have
$$
\varphi\left(
\frac{1}{e_X} [N_X^*\cL] \right) = 
\phi_{\cL}^{-1}\phi_{\frac{1}{e_X}N_X^*\mathcal{L}}=
\frac{1}{e_X}\phi_{\cL}^{-1}\widehat{N}_X\phi_{\mathcal{L}}N_X=
\frac{1}{e_X} N_X' N_X=
\frac{1}{e_X}N_X^2=N_X.
$$
This gives the first equation, since $\varphi$ is an isomorphism. The second equation follows from this 
and $\alpha_X=\frac{1}{e_X}\delta_X$.
\end{proof}

\begin{lem} \label{l2.2}
The map
$$
\mu_X: \End_\QQ(X) \ra \End_\QQ(A), \qquad \alpha \mapsto 
\frac{1}{e_X}\iota_X \alpha \psi_X \widehat \iota_X \phi_\cL
$$
is an injective homomorphism of $\QQ$-algebras with
$$
\mu_X(1_X) = \varepsilon_X.
$$
\end{lem}

\begin{proof}
The proof of the multiplicativity uses the fact that 
$\psi_X \widehat \iota_X \phi_\cL \iota_X =e_X 1_X$. The other assertions are obvious.
\end{proof}

Now consider a pair of complementary abelian subvarieties   
$\iota_{Y_i}^X: Y_i \hra X,\; i = 1,2,$ of the polarized abelian variety $(X,\cL|_X)$, i.e. 
with respect to the polarization $\cL|_X$.
Let $e_{Y_i}$ be the exponent of the polarization $(\cL|_X)|_{Y_i} = \cL|_{Y_i}$ and 
$\psi_{Y_i} = e_{Y_i} \phi_{\cL|_{Y_i}}^{-1}$. Let $N_X$ and $N_{Y_i}$ denote the norm 
endomorphisms of $X$ and $Y_i$ in $\End(A)$ with respect to the polarization $\cL$ and 
$\varepsilon_X$ and $\varepsilon_{Y_i}$ the corresponding symmetric idempotents of 
$\End_\QQ(A)$. 

\begin{lem} \label{l2.3}
With these notations we have
\begin{enumerate}
\item[(a)]  $e_X e_{Y_2} N_{Y_1} + e_X e_{Y_1} N_{Y_2} = e_{Y_1} e_{Y_2} N_X$,
\item[(b)] $\varepsilon_{Y_1} + \varepsilon_{Y_2} = \varepsilon_X$.
\end{enumerate}
\end{lem}

\begin{proof}
Let $N_{Y_i}^X \in \End(X)$ denote the norm endomorphism of $Y_i$ with respect to the 
polarization $\cL|_X$ for $i = 1,2$. 
We have for $i=1,2$,
$$
\mu_X(N_{Y_i}^X) =  \frac{1}{e_X} \iota_X N_{Y_i}^X \psi_X \widehat \iota_X \phi_{\cL} 
=   \frac{1}{e_X} \iota_X \iota_{Y_i}^X \psi_{Y_i} \widehat \iota_{Y_i}^X \phi_{\cL|_X} \psi_X \widehat \iota_X \phi_\cL
=  \iota_{Y_i} \psi_{Y_i} \widehat \iota_{Y_i} \phi_{\cL} =  N_{Y_i}
$$
Now \cite[page 125, equation (4)]{bl} says 
$$
e_{Y_2}N_{Y_1}^X + e_{Y_1} N_{Y_2}^X = e_{Y_1} e_{Y_2} 1_X.
$$
Applying the map $\mu_X$ we get
$$
e_{Y_1}e_{Y_2}\frac{1}{e_X}N_X = e_{Y_1}e_{Y_2}\mu_X(1_X) = 
e_{Y_2}\mu_X(N_{Y_1}^X) + e_{Y_1}\mu_X( N_{Y_2}^X) =
e_{Y_2}N_{Y_1}+ e_{Y_1} N_{Y_2}.
$$
This gives (a). Dividing (a) by $e_{Y_1}e_{Y_2}e_X$ gives (b).
\end{proof}

Let $f:A \ra B$ be an isogeny of (unpolarized) abelian varieties. So $f^{-1}$ exists in
$\Hom_\QQ(B,A)$ and the abelian subvarieties of $A$ correspond bijectively to the abelian 
subvarieties of $B$ under the map $X \mapsto f(X)$. Moreover the map
$$
\widetilde f: \End_\QQ(B) \ra \End_\QQ(A), \qquad  b \mapsto f^{-1} b f
$$
is an isomorphism of $\QQ$-algebras.
Now let 
$$
f: (A,\cL_A) \ra (B,\cL_B)
$$ 
be an isogeny of polarized abelian varieties, i.e. 
$f^*\cL_B \equiv \cL_A$. Let $\varphi_A: \NS_\QQ(A) \ra \End_\QQ^s(A)$, respectively 
$\varphi_B: \NS_\QQ(B) \ra \End_\QQ^s(B)$ be the isomorphisms associated to the polarization 
$\cL_A$, respectively $\cL_B$. Then we have,

\begin{lem} \label{l4.4}
$\widetilde f$ restricts to an isomorphism of $\QQ$-vector spaces 
$\widetilde f: \End_\QQ^s(B) \ra \End_\QQ^s(A)$ and the following diagram commutes
$$
\xymatrix
{
\NS_\QQ(B)  \ar[r]^{\varphi_B} \ar[d]_{f^*}& \End_\QQ^s(B) \ar[d]^{\widetilde f}\\
\NS_\QQ(A) \ar[r]^{\varphi_A} & \End_\QQ^s(A)
}
$$
\end{lem}

\begin{proof}
For the first assertion it suffices to show that if $b \in \End_\QQ(B)$ is symmetric with respect to 
$\cL_B$, then $\widetilde f(b)$ is symmetric with respect to $\cL_A$, since both $\QQ$-vector
subspaces are of the same dimension. This is a straightforward computation. For the commutativity of the diagram let $\beta \in \NS_\QQ(B)$. Then
$
\widetilde f \varphi_B(\beta) = f^{-1} \phi_{\cL_B}^{-1} \phi_{\beta} f =  
(\widehat f \phi_{\cL_B} f)^{-1} (\widehat f \phi_{\beta} f) = \phi_{f^*\cL_B}^{-1} \phi_{f^*(\beta)} = \varphi_{A} f^*(\beta).
$
\end{proof}

Consider the direct image map $f_*: \NS_\QQ(A) \ra \NS_\QQ(B)$ defined by the norm map 
of $f$. 

\begin{cor}
Let $X$ be an abelian subvariety of $A$ and $Y=f(X)$ with symmetric idempotents $\varepsilon_X$
and $\varepsilon_Y$. Then we have for $\alpha_X= \varphi_A^{-1}(\varepsilon_X)$ and 
$\alpha_Y=\varphi_B^{-1}(\varepsilon_Y)$
$$\alpha_X=f^*\alpha_Y$$
$$
f_*(\alpha_X) = \deg (f) \alpha_Y.
$$
\end{cor}

\begin{proof}
According to Lemma \ref{l4.4} we have $\alpha_X = f^*\alpha_Y$. This implies 
$$
f_*\alpha_X = f_*f^*\alpha_Y = \deg (f) \alpha_Y
$$
where the last equation can be seen, for example, in \cite[Proposition 21.10.18]{g}.
\end{proof}

\begin{rem}
The corollary can also be expressed in terms of the polarizations, since $deg(f) = 
\frac{\chi(\cL_A)}{\chi(\cL_B)}$.
\end{rem}

\section{The criterion in terms of the N\'eron-Severi group}

Recall that for any $g$-tuple of line bundles $\cL_i$ on a $g$-dimensional abelian variety $A$ an intersection number 
$(\cL_1 \cdots \cL_g)$ is defined. It extends in the usual way to give a rational intersection number 
$$
(\alpha_1 \cdots \alpha_g) \in \QQ
$$ 
for any elements $\alpha_i$ in $\NS_{\QQ}(A)$.
Let $\chi(\mathcal{L})$ be the Euler characteristic of $\mathcal{L}$; by Riemann-Roch this is just $d_1\cdots d_n$ if $\cL$ is of type $(d_1, \dots, d_g)$. Now let $X$ be an abelian subvariety 
of dimension $n$ of the polarized abelian variety $(A,\cL)$ of dimension $g$. Let $\delta_X$ denote 
the numerical class of $X$ in $\NS_\QQ(A)$ and $\alpha_X = \frac{1}{e_X} \delta_X$ as in Section 2. 

\begin{prop}\label{intersection} 
We have the following intersection numbers: 
$$
(\delta_X^r\cdot\mathcal{L}^{g-r})=\left\{\begin{array}{lcl} 
                                             e_X^r\chi(\cL) n! \prod_{i=n+1}^{g} (i-r)&for& 1\leq r\leq n,\\
                                             0&for& r \geq n+1.
                                             \end{array}\right.
$$
and
$$
(\alpha_X^r\cdot\mathcal{L}^{g-r})=\left\{\begin{array}{lcl} 
                                             \chi(\cL) n! \prod_{i=n+1}^{g} (i-r)&for& 1\leq r\leq n,\\
                                             0&for& r \geq n+1.
                                             \end{array}\right.
$$
\end{prop}

\begin{proof}
It suffices to prove the first equation.
Using Riemann-Roch and Lemma \ref{l2.1}, we have that for all $t\in\ZZ$,

\begin{eqnarray}\nonumber\deg(d_g(t-N_X))&=&\deg(d_g\phi^{-1}_\mathcal{L}\phi_{t[\mathcal{L}]-\delta_X})\\
\nonumber&=&\deg(d_g\phi^{-1}_{\mathcal{L}})\deg(\phi_{t[\mathcal{L}]-\delta_X})\\
\nonumber&=&\deg(d_g\phi^{-1}_{\mathcal{L}})\left(\frac{(t[\mathcal{L}]-\delta_X)^g}{g!}\right)^2.\end{eqnarray}
and 
$$
deg(d_g\phi_\mathcal{L}^{-1})=\frac{d_g^{2g}}{\chi(\mathcal{L})^2},
$$
since $d_g^{2g}
=\deg(d_g\cdot1_A)=
\deg(d_g\phi_{\mathcal{L}}^{-1}\phi_\mathcal{L})
=\deg(d_n\phi_{\mathcal{L}}^{-1})\chi(\mathcal{L})^2.$
On the other hand, according to \cite[Proposition 5.1.2]{bl} we have
$$
\deg(d_g(t-N_X))=d_g^{2g}t^{2g-2n}(t-e_X)^{2n}
$$ 
since $t^{2g-2n}(t-e_X)^{2n}$ is the 
characteristic polynomial of the rational representation of $N_X$. Together this gives, after dividing 
by $d_g^{2g}$ and taking square roots (since for sufficiently large $t$ both sides 
are positive)
$$
(t[\mathcal{L}]-\delta_X)^g=\chi(\cL)g! t^{g-n}(t-e_X)^n.
$$
Now expanding the polynomials and comparing the coefficients gives the result.
\end{proof}

Now let the notation be as in Lemma \ref{l2.3}. In particular $Y_1$ and $Y_2$ are complementary abelian subvarieties of $(X,\cL_X)$.
For $i=1,2$ denote $\alpha_i := \varphi^{-1}(\varepsilon_{Y_i}) \in \NS_\QQ(A).$ Then we have

\begin{prop}
$$
\sum_{j=1}^r {r \choose j}(\alpha_1^j \cdot \alpha_2^{r-j} \cdot \cL^{g-r}) = 
\left\{ \begin{array}{ccc} \chi(\cL) n! \prod_{i=n+1}^g (i-r) & for & r \leq n\\
                                      0 & for & r \geq n+1
           \end{array} \right.
$$
\end{prop}

\begin{proof} Let $\mu_X: \End_\QQ(X) \ra \End_\QQ(A)$ be the homomorphism of Lemma
\ref{l2.2}.
As complementary symmetric idempotents, $\varepsilon_{Y_1}^X$ and $\varepsilon_{Y_2}^X$
commute and so do their images $\varepsilon_{Y_1}$ and $\varepsilon_{Y_2}$ under the map 
$\mu_X$. This implies that 
$$
\varepsilon_X^r = (\varepsilon_{Y_1} + \varepsilon_{Y_2})^r = \sum_{j=0}^r {r \choose j} \varepsilon_{Y_1}^j \varepsilon_{Y_2}^{r-j}.
$$
Applying the map $\varphi^{-1}$, the assertion follows from Proposition \ref{intersection}.
\end{proof}

In \cite[Theorem 2.6]{robert}, it is shown that in the case of a principally polarized abelian variety 
the 
above intersection numbers characterize the classes that come from abelian subvarieties.  The problem when $\mathcal{L}$ is not a principal polarization, is that $\delta_X$ is not necessarily primitive. We can fix this, however, by considering symmetric idempotents instead of norm endomorphisms to get the following result.

\begin{thm}  \label{thm3.3}
Given a polarized abelian variety $(A,\cL)$ of dimension $g$.
The map $X\mapsto\alpha_X$ gives a bijection between the sets of
\begin{enumerate}
\item abelian subvarieties of dimension $n$ on $A$ and
\item classes $\alpha\in\mbox{NS}_{\QQ}(A)$ that satisfy
$$
(\alpha^r\cdot\mathcal{L}^{g-r})=\left\{\begin{array}{lcl}
                                            \chi(\mathcal{L}) n! \prod_{i=n+1}^{g}(i-r) &if&1\leq r\leq n\\                  
                                            0 &if& r \geq n+1.
                                            \end{array}\right.
$$
\end{enumerate}
\end{thm}

\begin{proof}
Let $X$ be an abelian subvariety of dimension $n$ of $A$. According to Proposition \ref{intersection}, 
$\alpha_X$ satisfies the equations in (2).
Conversely, let $\alpha\in\mbox{NS}_{\QQ}(A)$ satisfy the equations in (2). Let $q$ be a pöositive 
integer such that $q\varphi(\alpha)\in\mbox{End}(A)$. Then
\begin{eqnarray}\nonumber
\deg(t-q\varphi(\alpha))&=&\frac{1}{\chi(\mathcal{L})^2g!^2}((t[\mathcal{L}]-q\alpha)^g)^2\\
\nonumber&=&\left(\frac{1}{\chi(\mathcal{L})g!}\sum_{r=0}^g {g \choose r}(-q)^r
t^{g-r}(\alpha^r\cdot\mathcal{L}^{n-r})\right)^2\\
\nonumber&=&\left(\frac{1}{\chi(\mathcal{L})g!}\sum_{r=0}^n \frac{g!}{r! (g-r)!}(-q)^r
t^{g-r} \chi(\cL) n! (n+1-r) \cdots (g-r)\right)^2\\
\nonumber&=&\left( t^{g-n} \sum_{r=0}^n \frac{n!}{r!(n-r)!} (-q)^r t^{n-r}\right)^2 
= t^{2g-2n}(t-q)^{2n}
\end{eqnarray}
Since this is the characteristic polynomial of $q\varphi(\alpha)$ and $q\varphi(\alpha)$ is 
symmetric, then $q\varphi(\alpha)$ satisfies the equation $(q\varphi(\alpha))(q\varphi(\alpha)-
q)=0$ (see \cite[Lemma 2.7]{robert}). Therefore $\varphi(\alpha)^2=\varphi(\alpha)$, and 
so $\varphi(\alpha)$ is the symmetric idempotent associated to the abelian subvariety 
$X :=\mbox{image}(q\varphi(\alpha))$.
This completes the proof of the theorem, since clearly both maps are inverse to each other.
\end{proof}

Even if one finds many elliptic curves on an abelian variety, it is not always easy to tell whether the abelian variety decomposes as the product of these curves. For example, if $E$ and $E'$ are isogenous elliptic curves and $X$ is a simple abelian surface, then $E\times E'\times X$ has infinitely many elliptic curves, but does not split as the product of elliptic curves. Using the previous theorem, we can determine when this occurs. We first prove two lemmas.

\begin{lem}\label{l3.4}
If $(A,\mathcal{L})$ is a polarized abelian variety of dimension $g$ and $E_1,\ldots, E_g\subseteq A$ are elliptic curves with complementary abelian subvarieties $Z_1,\ldots,Z_g$, then the addition map 
$$
a:E_1\times\cdots\times E_g\to A
$$ 
is an isogeny if and only if the line bundle $\mathcal{O}_A(Z_1)\otimes\cdots\otimes\mathcal{O}_A(Z_g)$ is ample.
\end{lem}

\begin{proof} 
Set $\mathcal{S}:=\mathcal{O}_A(Z_1)\otimes\cdots\otimes\mathcal{O}_A(Z_g)$, let $X$ be the image of $a$ and let $Y$ be the complementary abelian subvariety of $X$. Clearly $Y\subseteq Z_i$ for every $i$ and $\mathcal{S}|_Y \simeq \cO_Y$. If $a$ is not an isogeny, i.e. 
$Y\neq\{0\}$, then $\mathcal{S}|_Y$ is not ample, and therefore $\mathcal{S}$ is not ample. 

Assume now that $a$ is an isogeny. We see that for all $I=\{i_1,\ldots,i_r\}\subsetneq\{1,\ldots,g\}$ and all $j\notin\{i_1,\ldots,i_r\}$, $Z_j$ intersects $Z_I:=Z_{i_1}\cap\cdots\cap Z_{i_r}$ properly.  Therefore, 
$$(\mathcal{S}^r\cdot\mathcal{L}^{n-r})=\sum_{\stackrel{I\subseteq\{1,\ldots,n\}}{\# I=r}}(\mathcal{L}^{n-r}|_{Z_I}),$$
which is positive since $\mathcal{L}$ is ample. By the Nakai-Moishezon Criterion for abelian varieties (\cite[Corollary 4.3.3]{bl}), we get that $\mathcal{S}$ is ample.
\end{proof}

\begin{lem} \label{l3.5}
If $E$ is an elliptic curve on $A$ with complementary abelian subvariety $Z$, then $\alpha_E$ is a 
positive multiple of $[\mathcal{O}_A(Z)]$, where $[\mathcal{O}_A(Z)]$ denotes the numerical class of the line bundle $\mathcal{O}_A(Z)$,
\end{lem}

\begin{proof}
By Lemma~\ref{l2.1} it is clearly sufficient to prove that $N_E^*\mathcal{L}$ is algebraically equivalent to a positive multiple of $\mathcal{O}_A(Z)$. 
Now according to \cite[Proposition 12.1.3]{bl} and since an elliptic curve is self dual, we can identify
$A/Z = \widehat E = E$ and the norm map $N_E$ factors as
$$
N_E: A \stackrel{p}{\ra} E \stackrel{\iota}{\hra} A.
$$
This gives $N_E^*\cL =p^*(\cL|_E)$. Since $E$ is an elliptic curve, we have  
$\mathcal{L}|_E\equiv m\mathcal{O}_{\tilde{E}}([0])$ for some positive integer $m$, where $[0]$ 
denotes the origin of ${E}$. Therefore, 
$$
N_E^*\mathcal{L}\equiv p^*\mathcal{L}|_E\equiv mp^*\mathcal{O}_{{E}}([0])\simeq m\mathcal{O}_A(Z),
$$
since $Z$ is the kernel of the projection $p$.
\end{proof}

\begin{cor}\label{split}
A polarized abelian variety $(A,\mathcal{L})$ of dimension $g$ splits isogenously as the product of elliptic curves if and only if there exist $\alpha_1,\ldots,\alpha_g\in\mbox{NS}_{\QQ}(A)$ such that
$$
(\alpha_i^r\cdot\mathcal{L}^{g-r})=\left\{\begin{array}{lcl}
                                            \chi(\mathcal{L}) \prod_{i=1}^{g}(i-r) &if&r=1\\                  
                                            0 &if& r \geq 2.
                                            \end{array}\right.
$$
and such that $\alpha_1+\cdots+\alpha_g$ is an ample $\QQ$-class.
\end{cor}

\begin{proof}
If $A$ splits up to isogeny as the product of sub elliptic curves $E_1,\ldots,E_g$, then, by Theorem \ref{thm3.3} and Lemma \ref{l3.4}, 
we obtain classes $\alpha_1,\ldots,\alpha_g$ that have the intersection numbers above. 
According to Lemma \ref{l3.5} we have 
$$
r_i\alpha_i=[\mathcal{O}_A(Z_i)]
$$ 
where $Z_i$ is the complementary abelian subvariety of $E_i$ in $A$ and  $r_i = \frac{1}{m_i}$
for some positive integer $m_i$. Define $r:=\max\{r_1,\ldots,r_g\}$.  Since $\delta= r_1\alpha_1+\cdots+r_g\alpha_g$ is ample by Lemma~\ref{l3.4}, then 
$$
r(\alpha_1+\cdots+\alpha_g)=\delta+(r-r_1)\alpha_1+\cdots+(r-r_g)\alpha_g
$$
is ample, being the sum of an ample class and a nef class. Therefore $\alpha_1+\cdots+\alpha_g$ is ample. 

Reciprocally, if we have $\alpha_1,\ldots,\alpha_g\in\mbox{NS}_{\QQ}(A)$ with the intersection numbers above, then Theorem \ref{thm3.3} implies that these come from elliptic curves $E_1,\ldots,E_g$ on $A$. By Lemma \ref{l3.5} there exists $r_i\in\mathbb{Z}_{>0}$ such that $r_i\alpha_i$ is the numerical cycle class of the complementary abelian subvariety of $E_i$. If $\alpha_1+\cdots+\alpha_g$ is ample, then $r_1\alpha_1+\cdots+r_g\alpha_g$ is ample (it lies within the ample cone), and by Lemma~\ref{l3.4}, $A$ splits as a product of the $E_i$.
\end{proof}

To determine when a $\QQ$-class is ample, we use the Nakai-Moishezon Criterion for abelian varieties (see \cite[Corollary 4.3.3]{bl}). This theorem states that if $\mathcal{L}$ is an ample line bundle on an abelian variety of dimension $g$ and $\alpha$ is a $\QQ$-divisor class, then $\alpha$ is ample if and only if $(\alpha^r\cdot\mathcal{L}^{g-r})>0$ for all $r=1,\ldots,n$. This, along with Corollary \ref{split}, gives us a numerical criterion for when an abelian variety splits as a product of elliptic curves.

\section{The criterion in terms of the period matrix}

Let $(A,\cL)$ be an abelian variety of dimension $g$ with polarization of type $(d_1, \dots, d_g)$. 
The N\'eron-Severi group of $A = V/\Lambda$ can be seen as a subgroup of $H^2(A,\ZZ)
\simeq  \wedge^2\ZZ^{2g}$. To be more precise, given a basis of $V$, let $\{\lambda_1,
\ldots,\lambda_g,\mu_1,
\ldots,\mu_g\}$ be a symplectic basis for $(A,\mathcal{L})$ such that with respect to these bases, 
$A$ is given by the period matrix $(D \; Z )$. If $x_1,\ldots,x_{2g}$ are the real 
coordinate functions of $\Lambda \otimes \RR$ associated to the basis of $\Lambda$ and $z_i$ the 
complex  
coordinate functions with respect to the basis of $V$, these functions are related by the equation
\begin{equation} \label{eq3.1}
\begin{pmatrix}z_1\\\vdots\\ z_g\end{pmatrix}=(D\;Z)\begin{pmatrix}x_1\\\vdots\\ 
x_{2g}\end{pmatrix}
\end{equation}
where $D=\mbox{diag}(d_1,\ldots,d_g)$. The global 1-forms $dz_1, \dots, dz_g$ form a $\CC$-
basis of the cotangent space of $A$ at 0. Since the forms 
$dz_i \wedge d\overline z_j, \; i,j = 1, \dots g$ generate the the $\CC$-vector space 
$H^{1,1}(A)$, this implies that 
$$
H^{1,1}(A) = \{ \omega \in H^2(A,\CC) \;| \; \omega \wedge dz_1 \wedge \cdots \wedge dz_g =0 \}
$$
Now the first Chern class induces an isomorphism 
$$
\NS_{\QQ}(A) \simeq H^2(A,\QQ) \cap H^{1,1}(A).
$$
Moreover,  $H^2(A,\QQ) \simeq \wedge^2 H^1(A,\QQ)$. So we can identify $H^2(A,\QQ)$ 
with respect to the above real basis as 
$$
H^2(A,\QQ) = \wedge^2 \QQ^{2g}.
$$
In other words, we consider $\{ dx_i \wedge dx_j \;|\; 1 \leq i < j \leq g \}$ as the canonical basis of $\wedge^2 \QQ^{2g}$.
So together we obtain the following identification
\begin{equation} \label{eq3.2}
\NS_{\QQ}(A) = \{\omega\in\wedge^2\QQ^{2n}:\omega\wedge dz_1\wedge\cdots\wedge dz_g=0\},
\end{equation}

With these identifications Theorem \ref{thm3.3} translates into the following theorem.

\begin{thm} \label{thm3.1}
Given a polarized abelian variety $(A,\cL)$ of dimension $g$ and type $(d_1,\dots,d_g)$.
The above identifications induce a bijection between the sets of
\begin{enumerate}
\item abelian subvarieties of $A$ of dimension $n$ and
\item differential forms $\omega\in\wedge^2\QQ^{2g}$ such that
\begin{enumerate}
\item $\omega\wedge dz_1\wedge\cdots\wedge dz_g=0$,
\item $\omega^{\wedge r}\wedge\theta^{\wedge (g-r)}=\left\{\begin{array}{ll}
\chi(\mathcal{L}) n! \prod_{i=n+1}^{g}(i-r) \cdot \omega_0&1\leq r\leq n\\0& r \geq n+1\end{array}\right.$\\
where $\theta$ is the first Chern class of $\cL$ and $\omega_0=(-1)^g dx_1\wedge dx_{g+1}\wedge\cdots\wedge dx_g\wedge dx_{2g}$.
\end{enumerate}
\end{enumerate}
\end{thm}

\begin{proof}
According to \cite[Exercise 2.6,(2b)]{bl}, the first Chern class of $\mathcal{L}$ corresponds to the differential form 
\begin{equation} \label{eq3.3}
\theta:=-\sum_{i=1}^nd_idx_i\wedge dx_{i+n}.
\end{equation}
Moreover, in these terms the intersection product in $\NS_{\QQ}(A)$ corresponds to the wedge product of differential forms. So the assertion is a translation of Theorem \ref{thm3.3}.
\end{proof}

\begin{exam}\label{Borowka}
In \cite[Theorem 8]{bor} a particular set of period matrices is shown that represents all principally polarized abelian varieties of dimension $g$ that contain an abelian subvariety of dimension $n$ and of type $(d_1,\ldots, d_n)$. Indeed, Bor\'owka obtains the family of matrices $Z=(z_{ij})_{i,j}$ where
$$\begin{array}{ll}z_{ij}=d_iz_{(g-n+i)j}&i,j=1,\ldots,n\\
z_{ij}=0&i=n+1,\ldots,g-n,j=1,\ldots,n\end{array}.$$
If $Z_n$ represents the principal $n\times n$ submatrix of $Z$, then putting $X=\mathbb{C}^n/(Z_n\mathbb{Z}^n+D\mathbb{Z}^n)$ with $D=\mbox{diag}(d_1,\ldots,d_n)$ and $A=\mathbb{C}^g/(Z\mathbb{Z}^g+\mathbb{Z}^g)$, we have an inclusion $X\hookrightarrow A$ whose analytic representation acts as
$$(x_1,\ldots,x_n)\mapsto (x_1,\ldots,x_n,0,\ldots,0,\frac{1}{d_1}x_1,\ldots,\frac{1}{d_n}x_n).$$
The rational representation of the norm endomorphism associated to $X$ in $A$ (with respect to the given polarization) is found to be
$$N=d_n\left(\begin{array}{cc}R&0\\0&R^t\end{array}\right)$$
where
$$R=\left(\begin{array}{cc}I_n&0\\0&0\\D^{-1}&0\end{array}\right)$$
(with the zero matrices having the appropriate size) and $D^{-1}=\mbox{diag}(d_1^{-1},\ldots,d_n^{-1})$. By \cite[Prop. 3.2]{robert}, the matrix of the alternating form associated to $X$ (with respect to the given symplectic basis) is 
$$M=\left(\begin{array}{cc}0&-I_g\\I_g&0\end{array}\right)N=d_n\left(\begin{array}{cc}0&-R^t\\R&0\end{array}\right).$$
Therefore, the differential form associated to the symmetric idempotent of $X$ is 
$$\omega_{d_1,\ldots,d_n}:=-\sum_{i=1}^n(dx_{i}\wedge dx_{g+i}+\frac{1}{d_i}dx_i\wedge dx_{g-n+i})$$
and the equations above on the period matrices can be simply represented by the equation $\omega_{d_1,\ldots,d_n}\wedge dz_1\wedge\cdots\wedge dz_g=0$. We then have that the moduli space of all principally polarized abelian varieties that contain an abelian subvariety of dimension $n$ and type $(d_1,\ldots,d_n)$ is
$$\mathcal{A}_g(\omega_{d_1,\ldots,d_n}):=\pi(\{Z\in\cH_g:\omega_{d_1,\ldots,d_n}\wedge dz_1\wedge\cdots\wedge dz_g=0\}),$$
where $\pi:\mathcal{H}_g\to\cA_g$ is the natural projection. This shows that our general approach to studying non-simple abelian varieties gives the same equations as those found in \cite{bor} when using a particular differential form. \qed
\end{exam}

Now we endow $\ZZ^{2g}$ with the canonical symplectic form and extend this form to the 
complex vector space $\CC^{2g}$. Let $x_1, \dots, x_{2g}$ denote the coordinate functions with 
respect to the canonical basis of
$\CC^{2g}$. We identify the vector space $\wedge^2 \CC^{2g}$ with the $\CC$-vector space of differential forms $\sum_{i<j} a_{ij} dx_i \wedge dx_j$ with constants $a_{ij} \in \CC$.

Let $(A,\cL)$ be a polarized abelian variety of dimension $g$ of type $(d_1,\dots,d_g)$ and suppose $A = V/\Lambda$. Given a basis of $V$ and a 
symplectic basis of $\Lambda$, we can identify
this basis with the canonical basis of $\ZZ^{2g} \subset \CC^{2g}$, such that the restrictions to the reals of the above coordinate functions $x_i$ (also denoted by $x_i$) are 
real coordinates of the complex vector space $V$. Let $(D \; Z)$ be the period matrix of $A$ with respect to the bases. We introduce complex coordinates $z_1, \dots, z_g$ of $V$
by means of equation \eqref{eq3.1}.

The matrix $Z$ is symmetric and positive definite and thus an element of the Siegel upper half 
space $\cH_g$ of rank $g$, which is a complex manifold of dimension ${g+1 \choose 2}$.
With respect to the real coordinates $x_i$  the first Chern class $\theta$ of $\cL$ is given by \eqref{eq3.3}.
With these notations we define for every positive integer $n < g$ a complex subvariety of 
$\wedge^2\CC^g \times \cH_g$ by
\begin{equation} \label{eq3.4}
\cE(g,n) := \{( \omega, Z) \in \wedge^2 \CC^{2g} \times \cH_g \;|\; \omega \; \mbox{satisfies 
equations (a) and (b) of Theorem \ref{thm3.1}} \}.
\end{equation}
If  we denote for any $Z \in \cH_g$ by $A_Z$ the polarized abelian variety defined by the period matrix $(D\; Z)$, 
then Theorem \ref{thm3.1} translates into to following corollary.

\begin{cor}
The abelian variety $A_Z$  admits an abelian subvariety of dimension $n$ if and only if there exists 
an $\omega \in \wedge^2 \QQ^{2g} \subset \wedge^2 \CC^{2g}$ such that $(\omega, Z) \in \cE(g,n).$ In other words,
$$
\{ Z \in \cH_g \;|\; A_Z \;\mbox{admits an abelian subvariety of dimension}\; n\} =
\cE(g,n) \cap (\wedge^2 \QQ^{2g} \times \cH_g).
$$
\end{cor}

The following proposition gives explicit equations for $g = 2$. 
For the principlally polarized case  see \cite[Thm. 3.1]{robert}, or \cite[Lemma 5.3]{Kani}.
Since any polarization is a multiple
of a primitive polarization, we may assume that the polarization of $A$ is of type $(1,d_2)$ with some positive integer $d_2$.

\begin{prop}\label{p3.3}
Let $(A,\mathcal{L})$ be a polarized abelian surface of type $(1,d_2)$ with period matrix 
$Z=\left(\begin{array}{cc}\tau_{1}&\tau_2\\\tau_2&\tau_3\end{array}\right)$. Then $A$ admits a 
sub elliptic curve if and only if there exists a vector 
$(a_{12},a_{13},a_{14},a_{23},a_{24},a_{34}) \in \QQ^6$ satisfying
\begin{eqnarray*}
-d_2&=&d_2a_{13}+a_{24},\\
 0&=&(\tau_1\tau_3-\tau_2^2)a_{12} - d_2a_{14}\tau_1 + d_2a_{13}\tau_2 - a_{24}\tau_2 + a_{23}\tau_3 + d_2a_{34} \; and \\
0&=&a_{14}a_{23}-a_{13}a_{24}+a_{12}a_{34}.\\
\end{eqnarray*}
\end{prop}
\begin{proof}
This comes from writing out the previous conditions. The first equation is (a) and the second and 
third equation (b) for $r=1$ and 2 of Theorem \ref{thm3.1}.
\end{proof}

Similarly the equations for $g=3$ can be given explicitly. Again we assume that the polarization is 
primitive, i.e. $d_1=1$. 

\begin{prop}  \label{p3.4}
Let $(A,\mathcal{L})$ be a principally polarized abelian threefold of type $(1,d_2,d_3)$ with period matrix $Z=(\tau_{ij})_{ij}$. Then $A$ admits a sub elliptic curve if and only if there exists a vector $(a_{ij})_{1\leq i<j\leq 6}\in\QQ^{15}$ such that
\newpage
\begin{eqnarray*}
 0&=&-a_{16}d_2d_3\tau_{11} + a_{14}d_2d_3\tau_{13} + a_{46}d_2d_3 - a_{26}d_3\tau_{12} - a_{36}d_2\tau_{13} +
a_{24}d_3\tau_{23} + \\
&& a_{34}d_2\tau_{33} - (d_3\tau_{12}\tau_{13} - d_3\tau_{11}\tau_{23})a_{12} - (d_2\tau_{13}^2 -
d_2\tau_{11}\tau_{33})a_{13} - (\tau_{13}\tau_{23} - \tau_{12}\tau_{33})a_{23}\\
0&=&-a_{16}d_2d_3\tau_{12} + a_{15}d_2d_3\tau_{13} + a_{56}d_2d_3 - a_{26}d_3\tau_{22} - a_{36}d_2\tau_{23} +\\
&&a_{25}d_3\tau_{23} + a_{35}d_2\tau_{33} - (d_3\tau_{13}\tau_{22} - d_3\tau_{12}\tau_{23})a_{12} - (d_2\tau_{13}\tau_{23} -
d_2\tau_{12}\tau_{33})a_{13} -\\ 
&&(\tau_{23}^2 -\tau_{22}\tau_{33})a_{23}\\
0&=&a_{56}d_2\tau_{13}
- a_{46}d_2\tau_{23} + a_{45}d_2\tau_{33} - ((\tau_{23}^2 - \tau_{22}\tau_{33})\tau_{11} - (\tau_{13}\tau_{23} - \tau_{12}\tau_{33})\tau_{12} +\\
&&(\tau_{13}\tau_{22} - \tau_{12}\tau_{23})\tau_{13})a_{12} - (d_2\tau_{13}\tau_{23} - d_2\tau_{12}\tau_{33})a_{14} + (d_2\tau_{13}^2 -
d_2\tau_{11}\tau_{33})a_{15} - \\
&&(d_2\tau_{12}\tau_{13} - d_2\tau_{11}\tau_{23})a_{16} - (\tau_{23}^2 - \tau_{22}\tau_{33})a_{24} + (\tau_{13}\tau_{23}
- \tau_{12}\tau_{33})a_{25} - (\tau_{13}\tau_{22} - \tau_{12}\tau_{23})a_{26}\\
0&=&-a_{15}d_2d_3\tau_{11} + a_{14}d_2d_3\tau_{12} + a_{45}d_2d_3 - a_{25}d_3\tau_{12} - a_{35}d_2\tau_{13} +
a_{24}d_3\tau_{22} + \\
&&a_{34}d_2\tau_{23} - (d_3\tau_{12}^2 - d_3\tau_{11}\tau_{22})a_{12} - (d_2\tau_{12}\tau_{13} -
d_2\tau_{11}\tau_{23})a_{13} - (\tau_{13}\tau_{22} - \tau_{12}\tau_{23})a_{23}\\
0&=&a_{56}d_2d_3\tau_{11} - a_{46}d_2d_3\tau_{12} + a_{45}d_2d_3\tau_{13} - ((\tau_{23}^2 - \tau_{22}\tau_{33})\tau_{11} -
(\tau_{13}\tau_{23} - \tau_{12}\tau_{33})\tau_{12} +\\
&& (\tau_{13}\tau_{22} - \tau_{12}\tau_{23})\tau_{13})a_{23} + (d_3\tau_{13}\tau_{22} - d_3\tau_{12}\tau_{23})a_{24}
- (d_3\tau_{12}\tau_{13} - d_3\tau_{11}\tau_{23})a_{25} + \\
&&(d_3\tau_{12}^2 - d_3\tau_{11}\tau_{22})a_{26} + (d_2\tau_{13}\tau_{23} -
d_2\tau_{12}\tau_{33})a_{34} - (d_2\tau_{13}^2 - d_2\tau_{11}\tau_{33})a_{35} +\\
&& (d_2\tau_{12}\tau_{13} -d_2\tau_{11}\tau_{23})a_{36}\\
 0&=&-a_{56}d_3\tau_{12} + a_{46}d_3\tau_{22} -
a_{45}d_3\tau_{23} - ((\tau_{23}^2 - \tau_{22}\tau_{33})\tau_{11} - (\tau_{13}\tau_{23} - \tau_{12}\tau_{33})\tau_{12} +\\ 
&&(\tau_{13}\tau_{22} -\tau_{12}\tau_{23})\tau_{13})a_{13} + (d_3\tau_{13}\tau_{22} - d_3\tau_{12}\tau_{23})a_{14} - (d_3\tau_{12}\tau_{13} - d_3\tau_{11}\tau_{23})a_{15} +\\
&&(d_3\tau_{12}^2 - d_3\tau_{11}\tau_{22})a_{16} - (\tau_{23}^2 - \tau_{22}\tau_{33})a_{34} + (\tau_{13}\tau_{23} - \tau_{12}\tau_{33})a_{35} -
(\tau_{13}\tau_{22} - \tau_{12}\tau_{23})a_{36}\\
d_2d_3&=&-a_{14}d_2d_3 - a_{36}d_3 - a_{25}d_3\\
0&=&-2a_{16}a_{34}d_2 - 2a_{13}a_{46}d_2 - 2a_{15}a_{24}d_3 - 2a_{12}a_{45}d_3 + (a_{36}d_2 +
a_{25}d_3)a_{14} +\\
&& (a_{14}d_3 + a_{36})a_{25} - 2a_{26}a_{35} + (a_{14}d_2 + a_{25})a_{36} -
2a_{23}a_{56}\\
0&=&(a_{36}a_{45} - a_{35}a_{46} + a_{34}a_{56})a_{12} - (a_{26}a_{45} - a_{25}a_{46} +
a_{24}a_{56})a_{13} +\\  
\nonumber&& (a_{26}a_{35} - a_{25}a_{36} + a_{23}a_{56})a_{14} - (a_{26}a_{34} -
a_{24}a_{36} + a_{23}a_{46})a_{15} +\\ 
\nonumber&&(a_{25}a_{34} - a_{24}a_{35} + a_{23}a_{45})a_{16} +
(a_{16}a_{45} - a_{15}a_{46} + a_{14}a_{56})a_{23} - \\
\nonumber&&(a_{16}a_{35} - a_{15}a_{36} +
a_{13}a_{56})a_{24} + (a_{16}a_{34} - a_{14}a_{36} + a_{13}a_{46})a_{25} -\\
\nonumber&& (a_{15}a_{34} -
a_{14}a_{35} + a_{13}a_{45})a_{26} + (a_{16}a_{25} - a_{15}a_{26} + a_{12}a_{56})a_{34} -\\
\nonumber&&(a_{16}a_{24} - a_{14}a_{26} + a_{12}a_{46})a_{35} + (a_{15}a_{24} - a_{14}a_{25} +
a_{12}a_{45})a_{36} +\\
\nonumber&& (a_{16}a_{23} - a_{13}a_{26} + a_{12}a_{36})a_{45} - (a_{15}a_{23} -
a_{13}a_{25} + a_{12}a_{35})a_{46} + \\
\nonumber&&(a_{14}a_{23} - a_{13}a_{24} +
a_{12}a_{34})a_{56}
\end{eqnarray*}
\end{prop}
\begin{proof}
For the computation we used the open source computer algebra system SAGE. The first 6 equations correspond
to the coefficient of $dx_1 \wedge \cdots \wedge \widehat {d x_i} \wedge \cdots \wedge dx_6$
(where $dx_i$ is omitted) of equation (a) of Theorem \ref{thm3.1}.  The last 3 equations 
correspond to $r=1,2,3$ in this order of equation (b) of Theorem \ref{thm3.1}.
\end{proof}

\section{Examples}

As pointed out in the introduction, we want to apply Propositions \ref{p3.3} and \ref{p3.4}
to factors of group algebra decompositions of some principally polarized abelian varieties 
in order to check whether they decompose further or not.

\subsection{The family $\cF_5$} In \cite{lrr} a 3-dimensional family $\mathcal{F}_5$ of 
principally polarized abelian varieties of dimension $5$ admitting an action of the dihedral group 
$D_5$ of order 10 was investigated. It is given by the
Riemann matrices
$$
Z=\left(\begin{array}{l l l l l} z_1&z_2&z_3&z_3&z_2\\ & z_1&z_2&z_3&z_3\\ & & z_1&z_2&z_3\\ 
& & & z_1&z_2\\ & & & & z_1\\ \end{array}\right)  \in \cH_5.
$$
Let $A = A_Z  = V/\Lambda_Z$ denote the abelian variety corresponding to $Z$. With respect to
 a basis of $V$ and a symplectic basis $\{\lambda_1,\dots,\lambda_5,\mu_1,\dots , \mu_5\}$ of 
$\Lambda_Z$ the period matrix of $A$ is 
$$
\Pi_Z =(I_5 \; Z).
$$
According to \cite{lrr} the group decomposition of $A_Z$ is 
$$
A_Z \sim E \times B^2.
$$
with an elliptic curve $E$ and an abelian surface $B$. We want to see whether $B$ decomposes 
further. Clearly the abelian subvariety $B$ of $A$ is not uniquely determined, but since the 
decomposability does not depend on the chosen subvariety in the equivalence class, we may assume acording to 
\cite[Proposition 4.4]{lrr} that the sublattice of $\Lambda$ defining $B$ has the basis
$$
\{\lambda_1-\lambda_5, \lambda_2-\lambda_4,\mu_1-\mu_5, \mu_2-\mu_4\}
$$
and moreover with respect to this basis, from the proof of \cite[Prop. 4.4]{lrr} it is deduced that the period matrix of $B$ with respect to the induced polarization of $A$ is given by
$$
\Pi_{B}=\left(\begin{array}{r r r r}2&0&2(z_1-z_2) & 2(z_2-z_3)\\ 0&2&2(z_2-z_3)&2(z_1-z_3)\\ \end{array}\right).
$$

\begin{prop} \label{p5.1}
{\em (a)} For a general $A \in \cF_5$ the abelian subvariety $B$ is irreducible.

{\em (b)} There is a union of surfaces in $\mathcal{F}_5$ whose members split 
isogenously as the product of elliptic curves.
\end{prop}

\begin{proof}
The abelian surface $B$ with period matrix $\Pi_B$ splits if and only if the isogenous principally 
polarized abelian surface $\widetilde B$ with period matrix 
$$
\frac{1}{2} \Pi_B  = \left(\begin{array}{r r r}
 z_1-z_2 & z_2-z_3\\ 
z_2-z_3&z_1-z_3\\ \end{array}\right).
$$ splits. According 
to Proposition \ref{p3.3}, $\widetilde B$ splits if and only if there exists a vector 
$(a_{12},a_{13},a_{14},a_{23},a_{24},a_{34})\in\QQ^6$ such that
\begin{eqnarray} \label{eq5.1}
\nonumber -1&=& a_{13} + a_{24}\\
0&=&[(z_1-z_2)(z_1-z_3) - (z_2-z_3)^2] a_{12} +\\
\nonumber&&+z_1(a_{23}-a_{14})+z_2(a_{13}+a_{14}-a_{24})+z_3(a_{24}-a_{13}-
a_{23})+a_{34}\\
\nonumber0&=&a_{14}a_{23}-a_{13}a_{24}+a_{12}a_{34}
\end{eqnarray}
Suppose that for a general $Z \in \cH_5$ the abelian surface $\widetilde B_Z$ splits. This implies that all 
coefficients of the quadratic polynomial in the $z_i$ vanish. So we get 
$a_{12} = a_{34} = 0, \; a_{23} = a_{14}$ and the system of equations
$$
a_{24} = -a_{13} -1, \qquad a_{24} = a_{13} + a_{14}, \qquad a_{14}^2 = a_{13} a_{24}.
$$
Eliminating $a_{14}$ and $a_{23}$ we remain with the equation $5a_{13}^2 + 5 a_{13} +1 = 
0$ which does not have a rational solution. This completes the proof of (a).

As for (b),  consider the integers 
$a_{12} =  a_{13} = a_{23} = a_{34} = 0, \; a_{24} = -1$   and $a_{14}:=a$ an arbitrary 
rational number. Then the first and the last equation of \eqref{eq5.1} are satisfied and the second 
equation defines the surface  
$$
S_a:=\{(z_1,z_2,z_3) \in \CC^3 \;|\; z_3 = (a+1)z_2-az_1\}
$$
We will show that there exist infinitely many $a\in\mathbb{Q}$ for which $S_a$ gives a non-empty family in $\cH_5$. In fact, denote by $y_i$ the imaginary part of $z_i$.
Then the eigenvalues of the imaginary part of the matrix $Z$ with $(z_1,z_2,z_3) \in S_a$ are:

$$
2a y_2+y_1+4y_2-2a  y_1,\quad \left( 1+\frac{a}{2}+\frac{\sqrt{5}a}{2} \right)  \left( y_1-y_2 \right) ,$$

$$ 
\left( 1+\frac{a}{2}-\frac{\sqrt{5}a}{2} \right)  \left( y_1-y_2 \right) ,\quad 
\left( 1+\frac{a}{2}+\frac{\sqrt{5}a}{2} \right)  \left( y_1-y_2 \right) ,\quad 
\left( 1+\frac{a}{2}-\frac{\sqrt{5}a}{2} \right)  \left( y_1- y_2 \right).
$$
The following conditions imply that these numbers are postive.
$$
\frac{-2}{\sqrt{5}+1}<a<\frac{2}{\sqrt{5}-1}, \quad y_1 > y_2.
$$
So for all rational numbers $a$ in that interval, $S_a$ defines a surface in $\cF_5$ all of whose 
abelian surfaces split up to isogeny.
\end{proof}

\begin{rem}
According to the proof of Proposition \ref{p5.1} there are infinitely many rational numbers $a$ for which $S_a$ defines a surface in 
$\cF_5$. We do not know how many of them are pairwise isomorphic. So we do not know 
how many such surfaces there are. 
\end{rem}

\subsection{The family $\cF_7$} In \cite{lrr} a 4-dimensional family $\mathcal{F}_7$ of 
principally polarized abelian varieties of dimension $7$ admitting an action of the dihedral group 
$D_7$ of order 14 was investigated. It is given by the
Riemann matrices
$$
Z=\left(\begin{array}{l l l l lll} z_1&z_2&z_3&z_4&z_4&z_3&z_2\\ & z_1&z_2&z_3&z_4&z_4&z_3\\ & & z_1&z_2&z_3&z_4&z_4\\ 
& & & z_1&z_2&z_3&z_4\\ 
& & & & z_1&z_2&z_3\\
&&&&&z_1&z_2\\
&&&&&&z_1 \end{array}\right)  \in \cH_7,
$$
Let $A = A_Z  = V/\Lambda_Z$ denote the abelian variety corresponding to $Z$. With respect to
 a basis of $V$ and a symplectic basis $\{\lambda_1,\dots,\lambda_7,\mu_1,\dots , \mu_7\}$ of 
$\Lambda_Z$ the period matrix of $A$ is 
$$
\Pi_Z =(I_7 \; Z).
$$
According to \cite{lrr}, the group algebra decomposition of $A_Z$ is 
$$
A_Z \sim E \times B^2.
$$
with an elliptic curve $E$ and an abelian threefold $B$. We want to see whether $B$ decomposes 
further. As before the abelian subvariety $B$ of $A$ is not uniquely determined, but since the 
decomposability does not depend on the chosen subvariety in the equivalence class, we may assume acording to 
\cite[Proposition 4.4]{lrr} that the sublattice of $\Lambda$ defining $B$ has the basis
$$
\{\lambda_1-\lambda_7, \lambda_2-\lambda_6,\lambda_3 - \lambda_5,\mu_1-\mu_7, \mu_2-\mu_6
,\mu_3 - \mu_5\}
$$
and moreover by the proof of \cite[Prop. 4.4]{lrr}, it is deduced that with respect to this basis the period matrix of $B$ with respect to the induced polarization of $A$ is given by
$$
\Pi_{B}=2\left(\begin{array}{r r r rrr}1&0&0&z_1-z_2 &z_2-z_3&z_3-z_4\\ 
0&1&0&z_2-z_3&z_1-z_4&z_2-z_4\\
0&0&1&z_3-z_4&z_2-z_4&z_1-z_3 \end{array}\right).
$$

\begin{prop}\label{p5.2}
For a general principally polarized abelian variety in $\cF_7$, $B$ is simple.
\end{prop}

\begin{proof} We need to show that a general $B$ contains no elliptic curves. Let $$
\omega=\sum_{i<j}a_{ij}dx_i\wedge dx_j\in\wedge^2\QQ^6.
$$ 
We see that if we divide the period matrix of $B$ by $2$, we obtain an isomorphic abelian variety which is principally polarized. After writing out the first six equations from Proposition~\ref{p3.4} for $\omega$ and the modified period matrix for $B$, we obtain polynomials in $z_1$, $z_2$, $z_3$ and $z_4$ with rational coefficients. If we take these to be algebraically independent, then each coefficient of the polynomials must be equal to 0, and we obtain the following relations:
$$a_{12}=a_{13}=a_{23}=a_{45}=a_{46}=a_{56}=0$$
$$a_{14}=-a_{15}+a_{25}-a_{34}$$
$$a_{16}=a_{34}$$
$$a_{24}=a_{15}$$
$$a_{26}=a_{35}=a_{15}+a_{34}$$
$$a_{36}=a_{25}-a_{34}$$
The parameters $a_{15}$, $a_{25}$ and $a_{34}$ are free, and the rest can be written in terms of these. If we now ask for $\omega\wedge\theta\wedge\theta=-dx_1\wedge dx_4\wedge dx_2\wedge dx_5\wedge dx_3\wedge dx_6$ (which is the 7th equation in \ref{p3.4} with $d_2=d_3=1$), we obtain the equation 
$$a_{15}=-\frac{1}{2}+3a_{25}-2a_{34}.$$
The equation $\omega\wedge\omega\wedge\theta=0$ leads to
$$42a_{25}^2 - 42a_{25}a_{34} + 14a_{34}^2 - 14a_{25} + 7a_{34} + 1=0.$$
Solving for $a_{25}$ in terms of $a_{34}$, we get that
$$a_{25}=\frac{42a_{34}+14\pm\sqrt{(42a_{34}+14)^2-168(14a_{34}^2+7a_{34}+1)}}{84}.$$
Now $(42a_{34}+14)^2-168(14a_{34}^2+7a_{34}+1)=-588a_{34}^2+28$. We see that $a_{34}=1/7$ makes this number a square, and using the stereographic projection, we get that this expression is the square of a rational number if and only if
$$a_{34}=-\frac{168+8t}{588+t^2}+\frac{1}{7}$$
for $t\in\mathbb{Q}$. In this case, we obtain the solutions
$$a_{25}=-\frac{2 \, t^{2}}{21 \, {\left(t^{2} + 588\right)}} - \frac{6 \,
t}{t^{2} + 588} - \frac{84}{t^{2} + 588} + \frac{2}{7}$$
$$a_{25}=\frac{2 \, t^{2}}{21 \, {\left(t^{2} + 588\right)}} - \frac{2 \,
t}{t^{2} + 588} - \frac{84}{t^{2} + 588} + \frac{4}{21}.$$
If we now ask for the equation $\omega\wedge\omega\wedge\omega=0$, we obtain the two equations
$$t^{6} - 756 \, t^{5} + 139356 \, t^{4} + 1481760 \, t^{3} - 52898832 \,
t^{2} - 261382464 \, t + 5489031744=0$$
$$3 \, t^{6} - 1092 \, t^{5} + 88788 \, t^{4} + 2140320 \, t^{3} -
29618736 \, t^{2} - 377552448 \, t + 3817474752=0.$$
A quick check in SAGE shows us that these two polynomials have no rational roots, and so the equations of Proposition~\ref{p3.4} have no rational solutions. 
\end{proof}

\begin{rem}
We note that the proof of Proposition~\ref{p5.2} shows that the Picard number of a very general $B$ in the above family is 3 (since the parameters $a_{15}$, $a_{25}$ and $a_{34}$ are free), and the Picard number of any $B$ is at least 3.
\end{rem}

\begin{rem}
Not every $B$ is simple. For example, setting $z_2=z_3=z_4$ we obtain a one dimensional family of threefolds that are isomorphic (as varieties) to the product of elliptic curves. We see that the differential forms
$$
\eta=\frac{1}{2}dx_1\wedge dx_4  \quad \mbox{and} \quad 
\nonumber\mu=\frac{1}{2}dx_2\wedge dx_5
$$
satisfy the equations of Proposition~\ref{p3.4} for this family, for example. By fixing $\eta$ and using the equations from Proposition~\ref{p3.4}, we see that the N\'eron-Severi group of a $B$ contains $\eta$ (and $\mu$) if and only if 
$$z_2=z_3=z_4.$$
This gives a surface in $\cF_7$ whose members split as the product of three elliptic curves and two abelian surfaces. 
\end{rem}

\subsection{The family $\cG$}

In this subsection we study the factors of the group algebra decomposition of a 3-dimensional family
of  principally polarized abelian varieties of dimension 6 with an action of the dihedral group of order $24$
$$
G:=\langle r,s:r^{12}=s^2=(rs)^2=1\rangle.
$$

First, recall some known results from \cite{brr}, and the references given there. Let $K$ be a finite group, an action of $K$ on a Riemann surface $X$ is an injective homomorphism from $K$ to the group of holomorphic automorphisms of $X$. $K$ acts with signature $(\gamma; m_1, \dots, m_t)$ if the quotient surface $X/K$ is of genus $\gamma$ and the covering $\pi:X\to X/K$ is ramified over $t$ points over which $\pi$ is locally $m_i:1$. Notice that the Riemann-Hurwitz equation
$$g=|K|(\gamma-1)+1+\frac{|K|}{2}\sum_{i=1}^t\left(1-\frac{1}{m_i}\right)$$
must be satisfied, where $g$ is the genus of $X$. This imposes restrictions on $\gamma$ and the $m_i'$s for given $K$ and $g$. A signature for $K$ satisfying this numerical condition is called an admissible signature. 
Moreover, given an admissible signature $(0; m_1, \dots, m_t)$ for a group $K$, a covering where $K$ acts with that signature can be constructed by a {\it generating vector}, which is a $t$-tuple $(g_1, \dots, g_t)$ of elements of $K$ generating $K$ and satisfying $g_1 \cdots g_t = 1$ such that $g_i$ is of order $m_i$ for all $i$. 

In fact, denote by $\Gamma$ the fundamental group of $\mathbb{P}^1$ with $t$ points removed.
Choose a set of generators $\alpha_1, \dots, \alpha_t$ of $\Gamma$ satisfying $\alpha_1 \cdots \alpha_t = 1$. Then a generating vector induces an exact sequence
$$
1 \to \Gamma' \to \Gamma \to K \ra 1
$$
and the normal subgroup $\Gamma'$ of $\Gamma$ defines the Galois covering  $\pi: X \to \mathbb{P}^1$.

Moreover, in \cite{brr} there is an algorithm to compute a symplectic basis of $H_1(X,\mathbb{Z})$ together with 
the induced action of $K$ on this basis, given a generating vector for the action (under the assumption of having $\mathbb{P}^1$ as total quotient). The algorithm starts by finding a basis that reflects the action, which is turned into a symplectic basis by the Frobenius algorithm, giving in this way a symplectic representation for the group $K$. This algorithm is implemented in a computer program, which is available at http://www.geometry.uchile.cl.

For the group studied in this subsection we have that the tuple $(r^3,r^4,r^5s,s)$ is a generating vector for the signature $(0;4,3,2,2)$. 
By the Riemann-Hurwitz formula, the corresponding curve is of genus 6.
Therefore there is a one-dimensional family of Jacobians with the action of $G$ coming from its action on the corresponding Riemann surface. The parameter of this family comes from the fact that the covering ramifies over four points on $\mathbb{P}^1$, therefore fixing three of them and allowing one to move we obtain the family.

Using the method described in \cite{brr} we compute the corresponding symplectic representation of $G$, and we compute the Riemann matrices in $\cH_6$ fixed by this action.
We obtain a $3-$dimensional family $\mathcal{G}$ of principally polarized abelian varieties of dimension $6$ with Riemann matrices
$$
Z=\left(\begin{array}{r r r r r r} -2x-2y+2z &x-\frac{3}{2}z&0&x+2y&2x+y-2z& -x-y+z\\ \\ & 2y+2z& x &-x-2y-\frac{1}{2}z&-2x+2y+2z& x-y-z\\ \\ & & -2y+z&y-\frac{1}{2}z&-x+y-\frac{1}{2}z& x\\ \\ & & & z&-3y&y\\ \\ & & & & 3z&-\frac{3}{2}z\\ \\ & & & & &z\\ \end{array}\right)
$$
The three parameters are complex numbers such that the corresponding matrix has positive definite imaginary part. The period matrix of a generic element $A$ in $\mathcal{G}$ is $\Pi=(I_6 \; Z)$, after fixing a symplectic basis 
$$
\Gamma=\{\alpha_1,\dots,\alpha_6,\beta_1,\dots , \beta_6\}
$$
for the lattice of $A_Z$.
A generic element $A\in \mathcal{G}$ decomposes as $A\sim S\times F$, where $S$ is an abelian surface and $F$ is a 4-dimensional abelian variety. We note that for the parameters $x=y=0$ and $z=1$, the matrix
$$Z_0=\left(\begin{array}{r r r r r r} 2 &-\frac{3}{2}&0&0&-2&1\\ \\ &2& 0 &-\frac{1}{2}&2& -1\\ \\ & & 1&-\frac{1}{2}&-\frac{1}{2}& 0\\ \\ & & & 1&0&0\\ \\ & & & & 3&-\frac{3}{2}\\ \\ & & & & &1\\ \end{array}\right)$$
is positive definite and $\tau Z_0\in\cG$ for any $\tau\in\mathbb{H}$. The family $\cG$ is therefore non-empty and thus three-dimensional.

Following the method in \cite[Table 4]{lr_Arch}, we determine that the induced polarization on $S$ is of type $(2,2)$, and on $F$ is of type $(1,1,2,2)$. Besides, the lattices 
for $S$ and $F$ are generated by the following symplectic bases:
$$\Gamma_S=\{\alpha_2+\alpha_5-\alpha_6,\alpha_1-\alpha_2+\alpha_4-\alpha_5,\beta_1+\beta_2-\beta_6,\beta_4-\beta_5-\beta_6\},$$
$$\Gamma_F=\{\alpha_3,-\alpha_1+\alpha_2,2\alpha_1-\alpha_2-\alpha_5+\alpha_6,-\alpha_1+\alpha_2+\alpha_4+\alpha_5,\beta_3,\beta_2-\beta_5,
\beta_1+\beta_2+\beta_6,\beta_4+\beta_5+\beta_6\}.$$
$S$ and $F$ are isotypical factors; this is, each one is the image of a central idempotent in the group algebra $\mathbb{Q}[G]$. Therefore they have the action of $G$. 
Using the group algebra decomposition, it is possible to decompose both factors further:
$$
S\sim E^2, \text{ and } F\sim S_1\times S_2,
$$
where $E$ is an elliptic curve with induced polarization of type $(4)$, and $S_i$ is an abelian surface with induced polarization of type $(2,4)$. $S_i$ corresponds to a primitive idempotent in $\mathbb{Q}[G]$, therefore
using the group action we cannot decompose  further either of the surfaces decomposing $F$. 

From $\Gamma_F$, replacing each $\beta_i$ by the combination of $\alpha_j$ given by the Riemann matrix of $A$, we get the following period matrix for $F=V_F/L_F$. Recall that the basis for $V_F$ is
$$
\{\alpha_3,-\alpha_1+\alpha_2,\frac{2\alpha_1-\alpha_2-\alpha_5+\alpha_6}{2},\frac{-\alpha_1+\alpha_2+\alpha_4+\alpha_5}{2}\},
$$
in order to consider $F$ with its induced polarization.
Under this consideration the period matrix of $L_F$ is
$$P_F=\left(\begin{array}{r r r r r r r r} 1&0&0&0& -2y+z&2x-y+\frac{z}{2}&2x &2y-z\\ 0&1&0&0&2x-y+\frac{z}{2}& 4x-2y+z&2x-2y+z&-2x+2y-z\\ 0&0&2&0&
2x&2x-2y+z&-4y+2z&2y-z\\ 0&0&0&2&2y-z& -2x+2 y-z &2y-z&-4 y-2z\\ \end{array}\right).$$

Moreover, we explore the primitive idempotents in  $\mathbb{Q}[G]$ giving varieties decomposing $F$, and we find 12 of them, all defining abelian surfaces with the same type of polarization $(2,4)$. We choose two of these abelian surfaces; the advantage of the chosen ones is that they define an isomorphism from $S_1\times S_2$ to $F$.

Let $\Gamma_i$ be a symplectic basis for the lattice $L_i$ of $S_i=V_i/L_i$ for $i=1,2$. Then we have
$$
\Gamma_1=\{\alpha_1-\alpha_2,-\alpha_2-2\alpha_3+2\alpha_4+3\alpha_5-\alpha_6,\beta_1-\beta_2+\beta_3-\beta_4+\beta_5,\beta_4+\beta_5+\beta_6\},
$$
$$
\Gamma_2=\{-\alpha_3+\alpha_4+\alpha_5,\alpha_1-\alpha_2+\alpha_4+\alpha_5,\beta_2-\beta_3+\beta_4+\beta_6,\beta_1-\beta_2+2\beta_3+2\beta_5+\beta_6\}.
$$
Note that $\langle\Gamma_1\rangle+ \langle\Gamma_2\rangle=\langle\Gamma_F\rangle$, therefore there is an isomorphism defined by the addition 
$$
S_1\times S_2\to F,\;\; (s_1,s_2)\mapsto s_1+s_2.
$$
Considering that the induced polarization on $S_i$ is of type $(2,4)$, and replacing the curves $\beta_j$ by the corresponding combination of curves $\alpha_k$ given by the Riemann matrix of $A$, we get the following period matrix for the lattice of $S_1$ and $S_2$
$$
P=\left(\begin{array}{r r r r } 2&0 &-4x-6y+3z &4x+4y-2z\\ 0&4 & 4x+4y-2z & -4y+2z\\ \end{array}\right),
$$
We obtain the same period matrix for both subvarieties, therefore they are isomorphic.

\begin{prop}
There are polynomials $Q,R\in\mathbb{C}[x,y,z,a_{12},a_{13},a_{14},a_{23},a_{34}]$ such that in the family of surfaces
$$f:\mbox{Spec }\mathbb{C}[x,y,z,a_{12},a_{13},a_{14},a_{23},a_{34}]/(Q,R)\to \mbox{Spec }\mathbb{C}[a_{12},a_{13},a_{14},a_{23},a_{34}]/(R),$$
each fiber over a rational point of the base scheme corresponds to a (possibly empty) family of principally polarized abelian varieties 
in $\mathcal{G}$ that split isogenously as the product of elliptic curves. Explicitly, if $(x,y,z)$ belongs to a fiber of $f$, we replace these in the matrix $Z$; if 
this matrix lies in the Siegel half space it corresponds to a completely decomposable principally polarized abelian variety in $\mathcal{G}$.
\end{prop}

\begin{proof}
In order for a member of $\mathcal{G}$ to split as the product of elliptic curves, it is only necessary to make $S_1$ (and therefore $S_2$) split. Since $S_1$ with the restricted polarization is of type $(2,4)$, we can divide this polarization by 2 to obtain a polarization of type $(1,2)$. Writing the equations from Proposition \ref{p3.3} for when an abelian surface splits as the product of elliptic curves in this scenario and replacing $a_{24}=2(1-a_{23})$, the second equation becomes 
\begin{eqnarray}\nonumber Q&=&a_{12}(-4x^2+2y^2+\frac{1}{2}z^2-4xy+2xz-2yz)+4a_{13}(2x+2y-z)+a_{14}(4x+6y-3z)\\
\nonumber&&+a_{23}(-2y+z)+2a_{34}+4x+4y-2z.
\end{eqnarray}
and the third equation becomes
\begin{equation}
\nonumber R=2a_{13}^2+2a_{13}+a_{14}a_{23}+a_{12}a_{34}.
\end{equation}
These polynomials give us the family of surfaces above.
\end{proof}

\begin{cor}
If we put $a_{12}=0$, we obtain linear equations that give us families of surfaces in $\mathcal{G}$ that consist of completely decomposable principally polarized abelian varieties.
\end{cor}

For example, the rational point $(0,0,0,1,0)$ gives us the surface $4x+2y-z=0$, and for $\tau\in\cH_1,$ the matrix
$$\tau \cdot \left(\begin{array}{rrrrrr}6&-5&0&1&-6&3\\&8&1&-3&6&-3\\&&4&-2&-3&1\\&&&4&0&0\\&&&&12&-6\\&&&&&4\end{array}\right)$$
has positive definite imaginary part and corresponds to a point of this fiber.

\end{document}